\documentclass[10pt]{amsart}
\usepackage[utf8]{inputenc}
\usepackage[english]{babel}
\usepackage{graphicx}
\usepackage{amsthm}
\usepackage{amssymb}
\usepackage{stmaryrd}
\usepackage{csquotes}
\usepackage{amsmath,mleftright}
\usepackage{xparse}
\usepackage{braket}
\usepackage{hyperref}
\usepackage[style=numeric, sorting=nyt]{biblatex}
\usepackage{amsfonts}
\usepackage{dsfont}
\usepackage{enumerate}
\usepackage{enumitem}
\usepackage{tikz-cd}
\usepackage{tabularx}
\usepackage{siunitx}
\usepackage[top=1 in,bottom=1in, left=3cm, right=3cm]{geometry}
\usepackage{placeins}
\usepackage{xcolor}
\usepackage{todonotes}
\usepackage{url}
\usepackage{comment}
\usepackage{caption}
\usepackage{setspace}
\usepackage[bottom]{footmisc}

\newtheorem{proposition}{Proposition}[section]
\newtheorem{lemma}[proposition]{Lemma}
\newtheorem{theorem}[proposition]{Theorem}
\newtheorem{corollary}[proposition]{Corollary}
\newtheorem{conjecture}{Conjecture}[section]
\theoremstyle{definition}
\newtheorem{remark}[proposition]{Remark}
\newtheorem{definition}[proposition]{Definition}
\newtheorem{example}[proposition]{Example}

\DeclareMathOperator{\Aut}{Aut}
\DeclareMathOperator{\Ric}{Ric}

\DeclareMathOperator{\DF}{DF}

\DeclareMathOperator{\Ima}{Im}
\DeclareMathOperator{\Rea}{Re}
\DeclareMathOperator{\Lie}{Lie}

\DeclareMathOperator{\Fix}{Fix}

\newcommand{\w}{\omega} %kahler metric
\newcommand{\Ricw}{\text{Ric}\hspace{0.35mm}\omega} %ricci of kahler metric
\newcommand{\RicW}{\text{Ric}\hspace{0.35mm}\Omega} %ricci of big kahler metric
\newcommand{\lap}{\Delta} %laplacian
\newcommand{\del}{\partial} %del operator
\newcommand{\delbar}{\overline{\partial}} %delbar operator
\newcommand{\W}{\Omega} %big kahler metric

\newcommand{\testconf}{\left(\mathcal{X},\mathcal{A}\right)} %test configuration
\newcommand{\RicFS}{\pi^*\text{Ric}\hspace{0.4mm}\omega_{FS}} %ricci curvature of FS metric
\newcommand{\hFS}{h_{FS}} %fubini-study hamiltonian for usual S^1 action
\newcommand{\C}{\mathbb{C}} %complex numbers
\newcommand{\R}{\mathbb{R}} %real numbers
\newcommand{\proj}{\mathbb{P}} %projective space
\newcommand{\g}{{\mathfrak{g}}}  %Lie algebra
\newcommand{\acts}{\circlearrowright} %group action
\newcommand{\euler}{e}
\newcommand{\DFklt}{Z_{k,l}\testconf}

\newcommand{\X}{\mathcal{X}}
\newcommand{\Fklt}{\mathcal{F}^Z_{k,l}}
\newcommand{\A}{\mathcal{A}}

\renewcommand{\L}{\mathcal{L}}
\renewcommand{\O}{\mathcal{O}}
\newcommand{\Zhat}{\mathcal{Z}(\X,\A)}
\newcommand{\Aeq}{\A_{\text{eq}}}
\newcommand{\DFhat}{\mathcal{Z}_{k,l}(\X,\A)}
\newcommand{\muklt}{\mu_{k,l}}

\newcommand{\Fhat}{\mathcal{F}^Z_{k,l}}
\newcommand{\ceq}{c_1(\X/\proj^1)_{\text{eq}}}

\title{Equivariant localisation in the theory of $Z$-stability for Kähler manifolds}

\author{Alexia Corradini}
\email{alexia.corradini@ip-paris.fr}

\numberwithin{equation}{section}

\bibliography{referencespaper}

\begin{document}
	
\begin{abstract} 
		We apply equivariant localisation to the theory of $Z$-stability and $Z$-critical metrics on a Kähler manifold $(X,\alpha)$, where $\alpha$ is a Kähler class. 
		We show that the invariants used to determine $Z$-stability of the manifold, which are integrals over test configurations, can be written as a product of equivariant classes, hence equivariant localisation can be applied. 
		We also study the existence of $Z$-critical Kähler metrics in $\alpha$, whose existence is conjectured to be equivalent to $Z$-stability of $(X,\alpha)$. In particular, we study a class of invariants that give an obstruction to the existence of such metrics. Then we show that these invariants can also be written as a product of equivariant classes.
		From this we give a new, more direct proof of an existing result: the former invariants determining $Z$-stability on a test configuration are equal to the latter invariants related to the existence of $Z$-critical metrics on the central fibre of the test configuration. This provides a new approach from which to derive the $Z$-critical equation.
\end{abstract}
	
	\maketitle
	
\section{Introduction}
	
	Stability conditions are a central subject of interest in algebraic geometry, in part due to their important role in the construction of moduli spaces. Over the past two decades, the notion of Bridgeland stability \cite{bridgeland}, defined on any triangulated category and generalizing the simpler notion of slope stability of a vector bundle, has been extremely popular for this reason. A lot of information can be obtained on the birational geometry of a moduli space of Bridgeland stable objects by varying the stability condition itself \cite{geometry_moduli_spaces}, explaining why general notions of stability are so useful.

	For different reasons, the notion of K-stability for a polarised variety has also been a central subject of research over the past 30 years. Originally introduced in the context of Kähler geometry, it was constructed progressively by Tian \cite{tian-inventiones}, then in a more algebro-geometric fashion by Donaldson \cite{donaldson}. Its purpose was to give an obstruction to the existence of cscK metrics on a Kähler manifold, leading to the famous Yau--Tian--Donaldson conjecture \cite{yau_openproblems,tian-inventiones,donaldson}. The conjecture relates an algebro-geometric notion of stability (namely K-stability), to existence of solutions to a geometric partial differential equation (namely cscK metrics). Another example is the Kobayashi--Hitchin correspondence \cite{kobaya-hitchin}, relating the existence of hermitian Yang-Mills connections on a vector bundle to its slope stability.

	This duality has motivated Dervan--McCarthy--Sektnan \cite{Dervan2020} to develop a more general framework in which solutions to geometric partial differential equations are characterised by stability conditions. They define notions of stability modelled on a subset of Bridgeland stability conditions on the derived category of coherent sheaves. Roughly speaking, this subset arises from \textit{Bayer's polynomial stability conditions} \cite{Bayer2009}, and involves a choice of \textit{central charge}, denoted by $Z$. The corresponding stability condition is $Z$-stability. They then associate to a central charge a geometric partial differential equation involving the curvature of a connection on the vector bundle, solutions of which are called $Z$-critical connections. Their theory extends the Kobayashi--Hitchin correspondence by conjecturally relating the existence of $Z$-critical connections on a vector bundle to its $Z$-stability.

	Further work was done by Dervan \cite{Dervan2021}, this time on the category of polarised varieties. In this framework, given a polarised variety $(X,L)$, the natural metrics are Kähler metrics, hence one expects to get a geometric partial differential equation for metrics in $c_1(L)$. This then leads to a conjectured correspondence between $Z$-stability of a polarised variety, and existence of solutions in $c_1(L)$ to a geometric partial differential equation, called $Z$-critical metrics. This conjectured correspondence extends the Yau--Tian--Donaldson conjecture to a broader family of stability conditions, of which K-stability is a special case.

	In the case of a polarised variety $(X,L)$, for a central charge $Z$, the notion of $Z$-stability is defined analogously to K-stability using the framework of test configurations, which are $\C^*$-degenerations $(\X,\mathcal{L})$ of $(X,L)$. Given a central charge, the corresponding notion of $Z$-stability is determined by the positivity of the number 
	\begin{equation}\label{stable_inv}
		\Ima\left(e^{-i\varphi}Z(\X,\mathcal{L})\right)
	\end{equation}
	defined on the test configuration, where $\varphi$ is the argument of the central charge $Z(X,L)\in\C$, and $Z(\X,\mathcal{L})\in\C$ is an integral invariant of the test configuration. In turn, $Z$-critical metrics are defined as metrics in $c_1(L)$ satisfying $$\Ima\left(e^{-i\varphi}\tilde{Z}(X,\w)\right)=0,$$ where $\tilde{Z}(X,\w)$ is a function of the metric $\w\in c_1(L)$, also depending on the central charge $Z(X,L)$. The approach used by Dervan to derive this partial differential equation is quite indirect, and involves deriving the Euler-Lagrange equations for an associated energy functional. The function $$\tilde{Z}(X,\w)$$ obtained in this way contains more than just Chern-Weil representatives, but also a laplacian term, to which it is difficult to assign a geometric intuition.

	As mentioned before, this framework generalises that of K-stability and cscK metrics. In this special case, the invariant \eqref{stable_inv} is the Donaldson-Futaki invariant (up to a normalisation, as we will see), explicitly given by: $$\DF(\X,\mathcal{L})=\frac{n\mu}{n+1}\mathcal{L}^{n+1}-\mathcal{L}^n\cdot K_{\X/\proj^1},$$ where $\mu=\frac{-K_X\cdot L^{n-1}}{L^n}$, and $n=\text{dim}(X)$. The original interpretation given by Donaldson for this invariant was in terms of coefficients of weight polynomials for the induced action on $H^0(\X_0,\L_0^k)$. Our definition stems from a later, more geometric interpretation of $\DF(\X,\L)$ as an intersection number on $(\X,\L)$, which was first suggested by Wang \cite{Wang2012} and Odaka \cite{Odaka2013}.  
	
	We also recall one of the early obstructions to the existence of cscK metrics on any Kähler manifold $(X,\alpha)$, where $\alpha$ is a Kähler class, namely the \textit{Futaki invariant}, introduced by Futaki \cite{futaki}. It is defined as follows, for any $\w\in\alpha$:
	\begin{align}
		\mathcal{F}:\Lie(\Aut(X))&\rightarrow \R \notag\\
		V&\mapsto \int_Xf_V\left(S(\w)-\hat{S}\right)\w^n \notag
	\end{align}
	where $S(\w)$ is the scalar curvature of $\w$, $\hat{S}$ its average, and $f_V$ is a hamiltonian-type function which we define later. The Futaki invariant is independent of the choice of $\w\in\alpha$ \cite[Theorem 4.12]{gabor}, and hence vanishes if there exists a cscK metric in $\alpha$.
	
	Returning to the framework of polarised varieties, in the case where the central fibre $(\X_0,\mathcal{L}_0)$ of a test configuration $(\X,\mathcal{L})$ is smooth, Donaldson \cite{donaldson} proved the following equality between the Donaldson--Futaki invariant on $(\X,\mathcal{L})$, and the Futaki invariant on the central fibre $(\X_0,\mathcal{L}_0)$: 
	\begin{equation}\label{legendre_main_theorem}
		\DF(\X,\mathcal{L})=-\pi\mathcal{F}(\X_0,\mathcal{L}_0)(V_0),
	\end{equation}
	where $V_0$ is the generator of the $S^1$-action on $\X$ restricted to $\X_0$. His proof relies on his original interpretation of $\DF(\X,\L)$ in terms of coefficients of weight polynomials, which is different - although equivalent - to the interesection-theoretic definition we use here. 
	
	More recently, Legendre \cite{Legendre2020} has used equivariant localisation, which we introduce in Section \ref{equivariant_background}, to prove this equality. It had been known and suggested by Futaki in his original paper that the Futaki invariant can be written as a product of equivariant classes on a Kähler manifold. Legendre noticed that the Donaldson-Futaki invariant can also be written as a product of equivariant classes on a test configuration, in a very similar way. Applying localisation then tells us that the only contributions from this invariant are contained in the fixed locus of the $S^1$-action on $\X$, from which she recovers \eqref{legendre_main_theorem}. The fact that this proof uses the intersection-theoretic definition of $\text{DF}(\X,\L)$ instead of the weight polynomial interpretation gives a rich differential-geometric insight on the invariant.
	
	A natural question is thus whether this method can be applied to the more general notions of $Z$-stability, i.e. if the $Z(\X,\mathcal{L})$ can be written as an equivariant class product for any central charge $Z(X,L)$ on a polarised variety $(X,L)$. We answer this question positively in this paper, and use it to give a new proof of a result analogous to \eqref{legendre_main_theorem} in the theory of $Z$-stability, originally proved by Dervan \cite{Dervan2021}. 
	\begin{theorem}\label{main_theorem}
		Let $(\X,\L)$ be a smooth compact test configuration for a polarised variety $(X,L)$, and $Z(X,L)$ be a central charge. If the central fibre $(\X_0,\L_0)$ is smooth, then the following equality holds for any $\w\in c_1(\mathcal{L}_0)$:
		\begin{align}
			\Ima\left(e^{-i\varphi}Z(\X,\mathcal{L})\right)=-2\pi\Ima\left(e^{-i\varphi}\tilde{Z}(\X_0,\w)\right).
		\end{align}
	\end{theorem}
\noindent We refer the reader to Theorem \ref{central_fibre_theorem} for a more precise statement.
 
This approach using equivariant cohomology provides a new understanding of the relation between stability defined from a test configuration and geometric partial differential equations. In the latter, it makes the appearance of the laplacian term much more natural; it follows directly from the fact that if $h$ is a hamiltonian for some vector field $V$ with Kähler metric $\w\in\alpha$, then $$\Ric\w+\lap h$$ is equivariantly closed, as we will discuss. It also provides a framework for deriving partial differential equations directly from the corresponding algebro-geometric stability condition. That is, one can actually derive the $Z$-critical equation from the associated algebro-geometric invariant using localisation. Furthermore, this suggests a way of deriving the partial differential equation in the more complicated case of central charges involving higher Chern classes of $X$, which is still not known. One could hope that the localisation framework we provide here would still apply; then one is left with the problem of extending a Chern-Weil representative of $c_k(X)$ to an equivariantly closed form.

%We also mention that, even after applying equivariant localisation to the invariants we discuss, finding examples in which explicit computations can be made remains extremely difficult. In the case of the Futaki invariant, Tian used a different form of equivariant localisation to compute the Futaki invariant of $\proj^2$ blown up at a point \cite[Section 3.2]{Tian2000}, thereby proving that it does not admit cscK metrics. However, it is not obvious to extend this formula to our more general invariants, even in simple cases. For the algebro-geometric invariants defined on a test configuration, one could hope to obtain a more theoretical result, as was done by Legendre for the Donaldson-Futaki invariant of the deformation to the normal cone in \cite[Section 6]{Legendre2020}. 

\subsection*{Notations}
	
We will be using concepts and notations used by Dervan in his paper \cite{Dervan2021}. However, we will be translating from his algebro-geometric framework to fit our differential-geometric approach. Instead of a polarised variety $(X,L)$, we will consider a Kähler manifold $(X,\alpha)$, where in the projective case we imagine $\alpha$ to be $c_1(L)$. Intersection products will be replaced by products of the corresponding cohomology classes. Given an (equivariant or non-equivariant) class $\beta$ which we wish to integrate over $X$, we will very often write $\beta([X])$ simply by $\beta$.

\subsection*{Acknowledgements}
	
This project has been the subject of a first year Master's research intership, under the supervision of Ruadhaí Dervan at the University of Cambridge. I wish to warmly thank him for suggesting this idea, and for his insight and guidance throughout the process. I also thank Eveline Legendre for our exchanges about her paper, from which the idea for this project was born. I was funded by the Jacques Hadamard Mathematical Foundation, and by Dervan's Royal Society University Research Fellowship for the duration of this work.

\section{Preliminaries}

\subsection{Equivariant localisation}\label{equivariant_background}
	
In this section we introduce some of the basic definitions of equivariant cohomology, and state the Atiyah--Bott localisation formula. We refer to \cite{equiv_topol, notes_equiv, supersym}, for a more complete presentation. We will ultimately be interested in $S^1$-actions on a compact Kähler manifold, however in this subsection we will consider the more general situation of the action by a compact and connected Lie group $G$ on a complex (Kähler) manifold $(X,\w)$ of complex dimension $n$.
	
The study of the equivariant cohomology of the $G$-action $G\acts X$, written $H_G^*(X)$, dates back to Cartan and Weil in the early $1950$s, with significant developments by Atiyah and Bott in the $1980$s \cite{atiyah_bott}. It can be defined as the ordinary cohomology of the quotient $(X\times E)/G$, where $E$ is a contractible space on which $G$ acts freely (such a space always exists, see \cite[Proposition 3.1]{equiv_topol}). In the case where the $G$-action on $M$ is free, this coincides with the cohomology of $M/G$ \cite[Theorem 4.1]{equiv_topol}. We write:
\begin{equation}
	H_G^*(X):=H^*((X\times E)/G).\notag
\end{equation}

Cartan then developed a realisation of this cohomology in the language of differential forms, which is the one we will be interested in. Here we give some of the basic definitions:
\begin{definition}[Equivariant differential form]
	An \textit{equivariant differential form} is a $G$-equivariant map $$\alpha:\g\rightarrow \Omega^*(X),$$ with respect to the adjoint action on $\g$ and the pullback action on differentials forms.
\end{definition}
\begin{definition}[Equivariant differential]
	The \textit{equivariant differential} operator acts on maps $$\alpha:\g\rightarrow\Omega^*(X)$$ as follows:
	\begin{align}
		(d_{\text{eq}}\alpha)(V):=d\alpha(V)-\iota_V\alpha(V)\hspace{5mm}  \text{for all} \hspace{2mm} V\in\g. \notag
	\end{align}
\end{definition}
\begin{definition}[Equivariantly closed and exact forms]
	A map $$\alpha:\g\rightarrow \Omega^*(X)$$ is said to be \textit{equivariantly closed} if 
	$$d_{\text{eq}}\alpha(V)=0 \hspace{2mm}\text{for all}\hspace{0.5mm} V\in\g.$$
	It is said to be \textit{equivariantly exact} if there exists a map $\beta:\g\rightarrow\W^*(X)$ such that $$\alpha(V)=d_{\text{eq}}\beta(V)\hspace{2mm}\text{for all}\hspace{0.5mm} V\in\g.$$
\end{definition}
	
One can check that $d_{\text{eq}}^2$ vanishes on the space of equivariant differential forms \cite[Proposition 2.4]{intro_equiv}. Then it has been shown by Cartan that the associated chain complex computes the equivariant cohomology $H_G^*(X)$. A proof of this is given in \cite{equiv_deRham}, and a reprint of Cartan's original papers can be found in the book \cite{supersym}. 
	
\begin{example}\label{form1}
	Let $\mu$ be a moment map for a $G$-action $\phi:G\hookrightarrow\Aut(X)$ on the symplectic manifold $(X,\w)$, i.e. an equivariant map $\mu:X\rightarrow\g^*$ such that, for all $V\in\g$, $$\iota_V\w=-d\langle\mu, V\rangle,$$ where $\iota_V$ denotes the contraction with (the fundamental vector field) of $V$. Then it is immediate to check that the equivariant form 
	$$\w-\mu:V\in\g\mapsto\w-\langle\mu, V\rangle$$ is equivariantly closed.
\end{example}
	
\begin{example}\label{form2}
	Let $V$ be a hamiltonian vector field on $(M,\w)$, and $h$ a hamiltonian for its flow, i.e. $\iota_V\w=-dh$. We denote by $\Ric\w$ the Ricci form of $\w$, which is defined in section \ref{Z-metrics}. From the identity \cite[Lemma 28]{Szekelyhidi2012} $$\iota_V\Ric\w=d\lap h,$$ we immediatly get that $\Ric\w+\lap h$ is equivariantly closed. It is an equivariant differential form because $\Ric\w$ is invariant, and $\lap h$ is equivariant because $\lap$ is equivariant under isometries (as observed in \cite[Proposition 3.5]{McCarthy2022}).
\end{example}

We can now state the Atiyah--Bott localisation formula:
	
\begin{theorem}[Atiyah-Bott localisation formula, \cite{symp_topo, equiv_topol}]\label{localisation}
	Let $\alpha:\g\rightarrow\Omega^*(X)$ be an equivariantly closed equivariant form, i.e. $d_{\text{eq}}\alpha=0$. Write $\Fix_G(X)$ for the fixed locus of the $G$-action, written as a union of its connected components as $\Fix_G(X)=\bigsqcup_{\lambda\in\Lambda}Z_{\lambda}$. Then for any $V\in\g$, we have
	\begin{equation}
		\int_M\alpha(V)=(-2\pi)^n\sum_{\lambda\in\Lambda}\int_{Z_{\lambda}}\frac{\iota^*_{Z_{\lambda}}\alpha(V)}{\textbf{e}(N_{Z_{\lambda}}^M)(V)}, \notag
	\end{equation}
	where $N_{\lambda}^X$ is the normal bundle of $Z_{\lambda}$ inside $X$, and $\euler(N_{\lambda}^X)$ its equivariant Euler class.
\end{theorem}
\noindent For a discussion on the equivariant Euler class, see \cite[Section 5.3]{notes_equiv}.

\begin{remark}
	The theorem actually holds more generally for maps $$\alpha:\g\rightarrow\W^*(X)$$ which are equivariantly closed, but not necessarily equivariant.
\end{remark}
	
One can show that, in the case of an $S^1$-action, the connected components of $\Fix_{S^1}(X)$ are smooth even-dimensional submanifolds of $X$ \cite[§ 8.5.3]{supersym}. In particular, for any connected component $Z_{\lambda} \in \Fix_{S^1}(X)$ of complex dimension $n_{\lambda}$, and at any point $z\in Z_{\lambda}$, $N_{\lambda}^X$ locally splits into complex lines: $$N_{\lambda}^X\cong\bigoplus_{k=1}^{n_Z}L_k,$$ on which $S^1$ acts with eigenvalue $iw_k$, where $w_k\in\R$ is called the \textit{weight} of the action on $L_k$.
	
In the case where this splitting into line bundles is global, the equivariant Euler class can be written as follows, where $\del_{\theta}$ denotes a generator of $\text{Lie}(S^1)$:

\begin{equation}\label{euler_class}
	\euler(N_Z^X)(\del_{\theta})=\prod_{k=1}^{n-n_Z}\left(2\pi c_1(L_k)-w_k\right).
\end{equation}
	
\begin{remark}
	Because we will only be interested in $S^1$-actions, we will generally omit $\del_{\theta}$ as an input, and just write $\euler(N_Z^X)$ for $\euler(N_Z^X)(\del_{\theta})$.
\end{remark}
	
In what follows, we will also make good use of the following formula for $\lap h$ at any point $z\in Z_{\lambda}\in\Fix_{S^1}(X)$, where $h$ is a hamiltonian for the $S^1$-action:
\begin{equation}\label{laplacian_h}
	\lap h=-2\sum_{k=1}^{n-n_Z}w_k.
\end{equation}
\noindent A proof of this is given in \cite[Lemma 4.5]{Boyer2018}.

\subsection{K-stability and cscK metrics}
	
We work on a Kähler manifold $(X,\alpha)$, where $\alpha\in H^{1,1}(X,\mathbb{R})$ denotes the Kähler class. K-stability and, as we will later see, $Z$-stability, are defined using \textit{test configurations}, which we introduce here. We only consider compact test configurations, because we will ultimately be computing integrals over their total space, and use a differential-geometric definition. For the same reason, we will also only consider smooth test configurations. We will state the Yau-Tian-Donaldson conjecture in this framework, but wish to emphasise that to obtain a fully accurate statement, one should really allow singular test configurations. More details about compactifying test configurations can be found in \cite[Section 2]{Boucksom2016}.
	
\begin{definition}[Test configuration]\label{test_configuration}
	A \textit{test configuration} for a Kähler manifold $(X,\alpha)$ is:
	\begin{enumerate}[label=(\roman*)]
		\item a compact complex manifold $\testconf$ with a class $\A\in H^{1,1}(X,\mathbb{R})$ and a $\C^*$-action $\C^*\acts X$ that preserves $\A$; 
		\item a $\C^*$-equivariant surjective map $\pi:\X\rightarrow\proj^1$ such that $(\X,\A)$;
		\item a $\C^*$-equivariant biholomorphism $\psi:\mathcal{X}\backslash\pi^{-1}(0)\cong X\times(\proj^1\backslash\{0\})$ for the action on $X\times\proj^1$ that is trivial on the $X$ component, and is the usual $\C^*$-action on $\proj^1$, and satisfies $\text{pr}_2\circ\psi=\pi$;
		\item writing  $X_t:=\pi^{-1}(t)$ and $\A_t:=\iota_{X_t}^*\A$, we require that $\A_t=\alpha$ and $\mathcal{A}_0$ is a Kähler class (in particular $(\X,\A)$ is relatively Kähler with respect to $\pi$).
	\end{enumerate}
\end{definition}
	
\begin{remark}
	Note that, because $S^1$ is compact and because the $S^1$-action preserves $\mathcal{A}$, tone can always choose a metric $\W\in\A$ to be $S^1$-invariant for the underlying action induced by $S^1\hookrightarrow\C^*$, by averaging over the Haar measure. We assume this throughout.
\end{remark}
	
\begin{remark}\label{P1_action}
	Recall the following facts about the standard $S^1$-action on $\proj^1=\C\cup\{\infty\}$ with the Fubini-Study metric $\w_{FS}$, written (in polar coordinates on $\C$) $\w_{FS}=\frac{4r}{(1+r^2)^2}dr\wedge d\theta$. 
	Then the standard $S^1$-action
	is the one underlying the multiplicative $\C^*$-action, and its fundamental vector field is $\del_{\theta}$. 
	A normalised moment map for this action is is $h_{FS}=\frac{(r^2-1)}{(r^2+1)}$, its weight is $1$ on $T_0\proj^1$, and $-1$ on $T_{\infty}\proj^1$.
	Note that, with our conventions, $\Ric\w_{FS}=\w_{FS}$. However, in what follows we will continue writing $\Ric\w_{FS}$ to emphasise that it represents $c_1(\proj^1)$.
\end{remark}
	
In this context, we define the Donalsdon-Futaki invariant of a test configuration $\testconf$ for a Kähler manifold $(M,\w)$ \cite{Legendre2020}:
\begin{equation}
	\text{DF}\testconf:=\frac{n\mu}{n+1}\A^{n+1}-\left(c_1(\X)-\pi^*c_1(\proj^1)\right)\cup\A^n
\end{equation}
where $\mu:=\frac{c_1(X)\cup\alpha^{n-1}}{\alpha^n}$.

\begin{remark}
	This definition for a Kähler manifold is found in \cite{Dervan2016,zak}, and is not the original one by Donaldson \cite{donaldson} in terms of coefficients of weight polynomials. However, they coincide on compactifications of polarised test configurations (i.e. in the projective setting), as proven in \cite{Odaka2013,Wang2012}.
\end{remark}
	
\begin{definition}[K-polystability]
	We say that $(X,\alpha)$ is \textit{K-polystable} if $\text{DF}\testconf\geq0$ for all test configurations, with equality exactly when $(X_0,\mathcal{A}_0)\cong(X,\alpha)$.
\end{definition}
	
The Yau-Tian-Donaldson conjecture \cite{donaldson,tian-inventiones,yau_openproblems} relates such a notion of stability to the existence of constant scalar curvature Kähler (cscK) metrics on Kähler manifolds:
	
\begin{conjecture}[Yau-Tian-Donaldson, \cite{donaldson,tian-inventiones,yau_openproblems}]
	A Kähler manifold $(X,\alpha)$ admits a cscK metric in $\alpha$ if and only if $(X,\alpha)$ it is $K$-polystable.
\end{conjecture}
	
\begin{remark}
	As was mentioned at the beginning of this section, this formulation of the Yau-Tian-Donaldson conjecture is not exact. In fact, to characterise the existence of cscK metrics, one must consider singular test configurations. 
\end{remark}
	
Before this was conjectured, Futaki \cite{futaki} had introduced what is now known as the $\textit{Futaki invariant}$, which gives an obstruction to the existence of cscK metrics. On any Kähler manifold $(X,\w)$, we write $\mathfrak{h}:=\Lie(\Aut(X))$ (whose elements we often confuse with their fundamental vector field), $S(\w)$ for the scalar curvature, $\hat{S}$ for its average, and we define the Futaki invariant as the map:
\begin{align}\label{futaki}
	\mathcal{F}:\hspace{1mm}\mathfrak{h}&\rightarrow\C\\
	V&\mapsto\int_Mf_V\left(S(\w)-\hat{S}\right)\w^n. \notag
\end{align}
Here $f_V$ is the function, unique up to a constant, appearing in the Hodge decomposition of the holomorphic vector field $-J\tilde{V}$, where any vector field $Y$ decomposes as $$g(Y,\cdot)=\xi_H+df_Y+d^cs_Y$$ with $\xi_H$ harmonic \cite[§2.1]{gauduchon}. One can show that the Futaki invariant is independent of the choice of metric in the Kähler class \cite[Theorem 4.12]{gabor}, so its non-vanishing guarantees that the corresponding class contains no cscK metrics.

\begin{remark}\label{fh}
	Note that, in the case where $\tilde{V}$ is hamiltonian, then $f_V$ is a hamiltonian for $\tilde{V}$. In fact, $g(-J\tilde{V},\cdot)=-\w(\tilde{V},\cdot)=dh_{\tilde{V}}$ where $h_{\tilde{V}}$ is a hamiltonian, so $h_{\tilde{V}}$ and $f_{\tilde{V}}$ are equal up to a constant.
\end{remark}
	
\begin{remark}
	A different way to view the Futaki invariant is as the constant derivative along automorphic flows of the Mabuchi functional. More details on this point of view can be found in \cite[Section 4.3]{gabor}.
\end{remark}
	
Using Remarks $\ref{form1}$ and $\ref{form2}$, and taking $V\in\mathfrak{h}$ to be the generator of an $S^1$-action, one can show that $\mathcal{F}(V)$ can we written as a product of equivariant classes as $$\mathcal{F}(V)=-\frac{1}{n}[(\w-h)^n]\cup[(\Ricw+\lap h)],$$ meaning localisation can be applied.

\subsection{$Z$-stability and $Z$-critical Kähler metrics}
	
We recall next the notions of $Z$-stability and $Z$-critical metrics introduced by Dervan \cite{Dervan2021} in the context of polarised varieties, translating to a differential-geometric language to fit our approach, working with a Kähler manifold $(X,\alpha)$.

We recall from the introduction that the notions of $Z$-stability are inspired by certain Bridgeland stability conditions on a triangulated category. More precisely, it is modelled on those conditions originating from a \textit{central charge}, which we view as a choice of topological information on $(X,\alpha)$. To each central charge is associated a notion of $Z$-stability, along with corresponding geometric partial differential equation.
Then Dervan conjectures an analogue of the Yau--Tian--Donaldson conjecture for each central charge \cite[Conjecture 1.1]{Dervan2021}, namely that $Z$-(poly)stability should be equivalent to existence of $Z$-critical metrics. We follow the notations used in \cite{Dervan2021}, and also remain in the case of central charges only involving $c_1(X)=-K_X$, and not higher Chern classes, as was done there.

A central charge requires a choice of some topological information, which can be encoded as:
\begin{itemize}\label{topological_information}
	\item a \textit{stability vector} $\rho=(\rho_0,\dots,\rho_n)\in\C^{n+1}$, normalised such that $\Ima(\rho_n)=i$;
	\item an $\textit{auxiliary cohomology class}$, i.e. a class $\Theta\in\oplus_{k=0}^nH^{k,k}(X,\C)$ whose $(0,0)$-component is 1;
	\item a $\textit{polynomial Chern form}$, which is a polynomial in $c_1(X)$, $f(c_1(X))=\sum_ {k=0}^n a_kc_1(M)^k$ with normalisation $a_0=a_1=1$.
\end{itemize}
\noindent Given these parameters, one defines a central charge $Z(X,\w)$ as follows \cite[Definition 2.4]{Dervan2021}:
\begin{definition}[Central charge]\label{central_charge}
	$Z(X,\alpha):=\sum_{l=0}^n\rho_l(\alpha^l\cup f(c_1(X))\cup\Theta)$
\end{definition}
\noindent We will write $\varphi(X,\w):=\arg(Z(X,\w))$, and assume throughout that $Z(X,\alpha)\neq0$ so this is well defined. Given a central charge, one can associate both a notion of algebro-geometric stability, i.e. $Z$-stability, and a geometric partial differential equation to which the solutions are known as $Z$-critical Kähler metrics. 
	
In what follows, it will be useful to write the central charge is a linear combination of the $(\alpha^l\cup c_1(X)^k\cup\Theta)$ as 
\begin{equation}\label{lincomb_centralcharge}
	Z(X,\alpha)=\sum_{k,l=0}^n a_{k,l}(\alpha^l\cup c_1(X)^k\cup\Theta)
\end{equation}
with coefficients $a_{k,l}\in\C$ chosen so that this expression coincides with the Definition \ref{testconf_invariant}.
	
\begin{example}
	An interesting choice of central charge given in \cite[Example 2.25]{Dervan2021} is $$Z(X,\alpha)=-e^{-i\alpha}e^{-ic_1(X)}=-\sum_{j=0}^n(\alpha^j\cup c_1(X)^{n-j}).$$ Then the corresponding partial differential equation is $$\Ima\left(e^{-i\varphi}\tilde{Z}(X,\w)\right)=0,$$ where $$\tilde{Z}(X,\w)=-\sum_{j=0}^n\frac{(-i)^j}{j!(n-j)!}\left(\frac{\w^{n-j}\wedge\Ric\w^j}{\w^n}-\frac{j}{n-j+1}\lap\frac{\w^{n-j+1}\wedge\w^{k-1}}{\w^n}\right).$$ This can be viewed as a manifold analogue of the deformed hermitian Yang-Mills equation on a holomorphic line bundle.
\end{example}

\subsubsection{$Z$-stability} \label{Z-stability}
	
Analogously to K-stability, $Z$-stability of a Kähler manifold $(X,\alpha)$ is defined using test configurations. Given a central charge $Z(X,\alpha)$ on $(X,\alpha)$, we would like to extend it to $\testconf$. The only difficult step is to extend $\Theta$ to a auxiliary class $\Theta_{\X}$ on $\testconf$. Here we give the explicit construction. 

\begin{lemma}\label{extend_theta}
	Given an auxiliary cohomology class $\Theta$ on a Kähler manifold $(X,\alpha)$, and a test configuration $\testconf$ for $(X,\alpha)$, $\Theta$ defines an $S^1$-equivariant cohomology class $\Theta_{\X\backslash\X_0}$ on $\X\backslash\X_0$.
\end{lemma}
\begin{proof}
	The proof relies heavily on the $\C^*$-equivariant biholomorphism $\psi:\X\backslash\pi^{-1}(0)\rightarrow X\times(\proj^1\backslash\{0\})$ appearing in the definition of a test configuration, where the action on $X\times(\proj^1\backslash\{0\})$ is trivial on the $X$ component, and is the usual $\C^*$-action on $\proj^1\backslash\{0\}$. 
		
	For any choice of representative $\theta\in\Theta$, $\psi^*\text{pr}_1^*\theta$ is a closed form on $\X\backslash\X_0$, whose $(0,0)$ term is $1$. From this construction it also becomes apparent that $\iota_V\psi^*\text{pr}_1^*\theta=0$ for any $\theta\in\Theta$, where $V$ is the fundamental vector field of the $S^1$-action. Indeed, because $\psi$ is equivariant, and because of the action chosen on $X\times(\proj^1\backslash\{0\})$, the vector field $\psi_*V$ is parallel to $\proj^1$. Since $\text{pr}_1^*\theta$ vanishes along components parallel to $\proj^1$,
	$\iota_V(\psi^*\text{pr}_1^*\theta)=\iota_{\psi_*V}(\text{pr}_1^*\theta)=0$. This means that any $\theta\in\Theta$ defines an $S^1$-equivariant class $[\psi^*\text{pr}_1^*\theta]_{\text{eq}}$ on $\X\backslash\X_0$, which we denote $\Theta_{\X\backslash\X_0}$.
		
	It remains to show that the equivariant class $\Theta_{\X\backslash\X_0}$ constructed does not depend on the choice of $\theta\in\Theta$. In fact, if $\theta'=\theta+d\lambda$, we have that 
	\begin{align}
		\psi^*\text{pr}_1^*\theta'&=\psi^*\text{pr}_1^*\theta+\psi^*\text{pr}_1^*d\lambda \notag\\
		&=\psi^*\text{pr}_1^*\theta+d\psi^*\text{pr}_1^*\lambda \label{hi}. \notag
	\end{align}
	\noindent But $\iota_V\psi^*\text{pr}_1^*\lambda=0$, from the same argument used at the beginning of the proof to show that $\iota_V\psi^*\text{pr}_1^*\theta=0$ for any $\theta\in\Theta$. Then $d\psi^*\text{pr}_1^*\lambda=d_{\text{eq}}
	\psi^*\text{pr}_1^*\lambda$, and hence $$\psi^*\text{pr}_1^*\theta'=\psi^*\text{pr}_1^*\theta+d_{\text{eq}}\psi^*\text{pr}_1^*\lambda\hspace{2mm}\in \Theta_{\X\backslash\X_0}.$$ 
\end{proof}
	
In what follows, we will be requiring the following (stronger) corollary:
	
\begin{corollary}\label{corollary_extend_theta}
	Each term $\Theta_p\in H^{p,p}(X,\C)$ the decomposition $$\Theta=\oplus_{p=0}^n\Theta_p$$ defines an equivariant class $(\Theta_{\X\backslash\X_0})_p$ on $\X\backslash\X_0$.
\end{corollary}
\begin{proof}
	By the previous lemma, it follows that each $\theta_{\X\backslash\X_0}\in\Theta_{\X\backslash\X_0}$ constructed in the proof satisfies $\iota_V\theta_{\X\backslash\X_0}=0$. 
	This combined with the fact that it is closed is equivalent to saying that each term $(\theta_{\X\backslash\X_0})_p$ is equivariantly closed, hence we can consider their equivariant classes $(\Theta_{\X\backslash\X_0})_p$.
	It follows directly from the argument in the proof of the Lemma that this does not depend on the choice of $(\theta_{\X\backslash\X_0})_p$.
\end{proof}

Because we require an equivariant class on all of $\X$, we will have to restrict ourselves to auxiliary classes $\Theta$ on $(X,\alpha)$ such that each of the induced equivariant classes $(\Theta_{\X\backslash\X_0})_p$ contains some $(\theta_{\X\backslash\X_0})_p$ that extends smoothly to $\X$. Then we denote the equivariant class of the extension $(\theta_\X)_p$ of $(\theta_{\X\backslash\X_0})$ by $(\Theta_\X)_p$. This is well defined, because an equivariantly exact form on $\X\backslash\X_0$ that extends smoothly to $\X$ extends to an equivariantly exact form on $\X$.

Given a central charge, each test configuration carries the following invariant \cite[Definition 2.8]{Dervan2021}:
\begin{equation}\label{testconf_invariant}
	Z\testconf:=\sum_{l=0}^n\frac{\rho_l}{l+1}\left(\mathcal{A}^{l+1}\cup f(c_1(\mathcal{X})-[\RicFS])\cup\Theta_{\X}\right) \in\C. \notag
\end{equation}
	
\noindent Thus, for a given test configuration $\testconf$, any central charge can be written as a linear combination of the terms $$\DFklt:=\left(\mathcal{A}^{l+1}\cup (c_1(\mathcal{X})-[\RicFS])^k\cup\Theta_\X\right).$$ Explicitly, in terms of the coefficients $a_{k,l}$ of the central charge \eqref{lincomb_centralcharge}:
\begin{equation}\label{lincomb2}
	Z(\X,\mathcal{A})=\sum_{k,l=0}^n \frac{a_{k,l}}{l+1}\hspace{0.2mm} \DFklt
\end{equation}

\begin{remark}\label{resolution_of_indeterminacy}
	Our restriction to auxiliary classes $\Theta$ such that $\Theta_{\X\backslash\X_0}$ as defined by Lemma \ref{extend_theta} extends smoothly to an equivariant class $\Theta_\X$ on $\X$ could be avoided by passing through an equivariant resolution of indeterminacy of the form  
	\[
	\begin{tikzcd}
		\mathcal{Y} \arrow[swap]{d}{q} \arrow{dr}{r} &  \\
		X\times\proj^1 \arrow[dotted]{r}{} & \X.
	\end{tikzcd}
	\]
	This process is used and described by Dervan in \cite[Section 2.1]{Dervan2021}, and \cite[Lemma 2.3]{Dervan2021} shows that the result is independent of the choice of resolution. However, through this construction, the central fibre $\mathcal{Y}_0$ will \textit{a priori} be singular, which would not allow us to produce an analogue of our main Theorem \ref{main_theorem}.
\end{remark}

The following plays the role of a $Z$-critical analogue of the Donaldson-Futaki invariant:
\begin{equation}\label{Zstable_quantity}
	\Ima\left(e^{-i\varphi}Z(\X,\A)\right).
\end{equation}
Then, similarly as for K-polystability, we define $Z$-polystability in the following way:
	
\begin{definition}[$Z$-polystability]\label{Z-polystability}
	We say that a Kähler manifold $(X,\alpha)$ is \textit{$Z$-polystable} if, for any test configuration $\testconf$ for $(X,\alpha)$, $\Ima\left(e^{-i\varphi}Z(\X,\A)\right) \geq0$, with equality exactly when $(\X_0,\mathcal{A}_0)\cong(X,\alpha)$.
\end{definition}
	
\begin{remark}
	Note that the original definition given by Dervan \cite{Dervan2021} considers $$\Ima\left(\frac{Z(\X,\A)}{Z(X,\alpha)}\right)$$ instead of $$\Ima\left(e^{-i\varphi}Z(\X,\A)\right).$$ These, of course, give rise to equivalent characterisations of $Z$-stability, since $$e^{-i\varphi}Z(X,\alpha)\in\mathbb{R}$$ is strictly positive by our assumption that $Z(X,\alpha)\neq0$.
\end{remark}
	
A useful equality that follows from the definitions is sometimes called \textit{translation invariance}, and tells us that the quantity $$\Ima\left(e^{-i\varphi}Z(\X,\A)\right)$$ remains unchanged if we change $\mathcal{A}$ by addition of pullbacks of the Fubini-Study metric on $\proj^1$. Explicitely, for any $m$, \cite[Lemma 2.13]{Dervan2021} proves that
\begin{equation}\label{translation_invariance}
	\Ima\left(e^{-i\varphi}Z(\X,\A+ c_1(\mathcal{O}(m))\right)=\Ima\left(e^{-i\varphi}Z(\X,\A)\right).
\end{equation}
	
\begin{remark}
	In algebro-geometric language, if $\mathcal{A}=c_1(\mathcal{L})$, where $\mathcal{L}$ is a relatively ample line bundle, then \eqref{Zstable_quantity} is unchanged if one changes the polarisation of $\X$ by the operation $\mathcal{L}\rightarrow\mathcal{L}+\pi^*\mathcal{O}(k)$ for any $k$.
\end{remark}

\begin{remark}
	An axiomatic approach to $Z$-stability, more in line with Bridgeland stability, is developed by Dervan in \cite{Dervan2022}.
\end{remark}
	
\subsubsection{$Z$-critical metrics} \label{Z-metrics}
	
We start by recalling that the \textit{Ricci curvature} of a Kähler metric $\w$ is defined as $$\Ric\w:=-i\del\delbar\log(\det(\w)).$$
Its cohomology class is the first Chern class of $X$, $c_1(X)=[\Ric\w]$, and is independent of $\w$ \cite[Page 14]{gabor}.
	
To a central charge, in addition to a notion of stability, Dervan also associates a geometric partial differential equation, which will be used to define the notion of $Z$-critical metric. Fixing some $\theta\in\Theta$ as before, one associates to each term $\alpha^l\cup c_1(X)^k\cup\Theta$ appearing in \eqref{central_charge} the function \cite[Equation 2.1]{Dervan2021}:
\begin{equation}
	\frac{\Ricw^k\wedge\w^l\wedge\theta}{\w^n}-\frac{k}{l+1}\lap\left(\frac{\Ricw^{k-1}\wedge\w^{l+1}\wedge\theta}{\w^n}\right). \notag
\end{equation}
Then, writing the central charge as a linear combination as in Equation \eqref{lincomb_centralcharge}, the function $\tilde{Z}(X,\w)$ is written as 
\begin{equation}\label{yourmom}
	\tilde{Z}(X,\w)=\sum_{k,l=0}^na_{k,l}\left(\frac{\Ric\w^k\wedge\w^l\wedge\theta}{\w^n}-\frac{k}{l+1}\lap\frac{\Ric\w^{k-1}\wedge\w^{l+1}\wedge\theta}{\w^n}\right).
\end{equation}
Clearly, this function is satisfies
\begin{equation}
	\int_X\tilde{Z}(X,\w)\w^n=Z(X,\w)\notag
\end{equation}

From this, we obtain the following definition, which is equivalent to requiring that the argument of $\tilde{Z}(X,\alpha)$ is constant:
	
\begin{definition}\label{Z-critical}
	A metric $\w$ is \textit{$Z$-critical} if it satisfies the equation
	\begin{equation}
		\Ima(e^{-i\varphi(X,\w)}\tilde{Z}(X,w))=0,
	\end{equation}
	and the positivity condition $\Rea(e^{-i\varphi(X,\w)}\tilde{Z}(X,w))>0$ holds.
\end{definition}
	
This is a partial differential equation on the space of Kähler metrics in the Kähler class $\alpha$, or equivalently on the space of Kähler potentials. 
	
Given a choice of $\theta\in\Theta$, we now introduce the invariants that appear explicitly in \cite[Corollary 3.12]{Dervan2021}, and play the role of $Z$-critical analogues of the Futaki invariant:
\begin{align}\label{Z-futaki}
	\mathcal{F}^Z:=\hspace{2mm}&\mathfrak{h}\rightarrow \C \\
	&V\mapsto \int_X f_V \Ima(e^{-i\varphi}\tilde{Z}(X,\alpha))\w^n, \notag
\end{align}
where $f_V$ is the same function appearing in \eqref{futaki}.
	
Recall from Remark \eqref{fh} that if $V\in\mathfrak{h}$ generates an $S^1$-action, and requiring the normalisation $\int_Xf_V\w^n=0=\int_Xh_V\w^n$, then $f_V=h_V$. In this case, \cite[Corollary 3.12]{Dervan2021} shows that the $\mathcal{F}^Z$ are independent of the choice of $\w$ in the Kähler class $\alpha$ (as the notation suggests). In this sense, these invariants represent obstructions to the existence of $Z$-critical metrics in $\alpha$.
It will be useful to express the $\mathcal{F}^Z$ as linear combinations over $k$ and $l$ using Equation \eqref{yourmom}, as
\begin{equation}\label{miam}
	\mathcal{F}^Z=\sum_{k,l=0}^na_{k,l}\mathcal{F}^Z_{k,l},
\end{equation}
	where 
\begin{align}\label{Zkl-futaki}
	\mathcal{F}^Z_{k,l}:=\hspace{2mm}&\mathfrak{h}\rightarrow \C \notag \\
	&V\mapsto \int_X f_V \Ima\left(e^{-i\varphi}\left(\frac{\Ric\w^k\wedge\w^l\wedge\theta}{\w^n}-\frac{k}{l+1}\lap\frac{\Ric\w^{k-1}\wedge\w^{l+1}\wedge\theta}{\w^n}\right)\right)\w^n. \notag
\end{align}

\section{Localisation}

\subsection{Localisation and $Z$-stability}\label{localizing_Z-stability}
	
Let us fix a central charge $Z(X,\alpha)$ on $(X,\alpha)$. Our main goal in this section will be to write the associated central charge of a test configuration $Z(\X,\A)$ as a product of equivariant classes, then to apply Theorem \ref{localisation}. The fact that this can be done implies that the only contributions to $Z(\X,\A)$ come from the fixed locus of the $S^1$-action on $\mathcal{X}$. 
	
In the framework of a test configuration used to define $Z$-stability, the $S^1$-action we are interested in is the one underlying the $\C^*$-action of $\X$.
	
Because $\pi:\mathcal{X}\rightarrow\proj^1$ is $\C^*$-equivariant, the fixed locus must be contained in $\X_0$ and $\X_{\infty}$, and the definition of a test configuration \ref{test_configuration} tells us that all of $\X_{\infty}$ is fixed. 
A crucial part of our main results in this subsection will be to prove that the contributions from the fibre at infinity are trivial when it comes to determining $Z$-stability. This is not immediately apparent in our definition of $Z(\X,\A)$. However, it can be made explicit by considering a certain normalisation $$Z(\X,\A)\rightarrow Z(\X,\A)+\lambda\A^{n+1}$$ for some $\lambda\in\C$, for which the contribution from the fibre at infinity will exactly vanish. Clearly, this has to be done in a way that preserves $$\Ima\left(e^{-i\varphi}Z(\X,\A)\right),$$ as we would like the inequality defining $Z$-polystability \ref{Z-polystability} to be unchanged.
We also require that, after this modification, the translation invariance from Equation \eqref{translation_invariance} still holds. This normalisation is made explicit in the following lemma.
	
\begin{lemma}\label{renormalisation}
	Given a central charge on a Kähler manifold $(X,\alpha)$, the quantity 
	\begin{equation}
		\mathcal{Z}(\X,\A):=Z\testconf -\frac{\mathcal{A}^{n+1}}{(n+1)\alpha^n}Z(X,\alpha) \notag
	\end{equation}
	is translation invariant, meaning it is invariant under the transformation $\A\rightarrow\A+c_1(\mathcal{O}(m))$ for any $m$. Furthermore, it satisfies 
	\begin{equation}\label{y}
		\Ima\left(e^{-i\varphi}\mathcal{Z}(\X,\A)\right) =\Ima\left(e^{-i\varphi}Z(\X,\A)\right)  \notag
	\end{equation}
\end{lemma}
\begin{proof}
	The last equality \eqref{y} follows immediately from the fact that $\frac{\mathcal{A}^{n+1}}{(n+1)\alpha^n}$ is real. To prove the first equation \eqref{renormalisation}, we take $\A\rightarrow\A+c_1(\O(m))$ and compute. We use the following equality from \cite[Lemma 2.13]{Dervan2021}: $$Z(\X,\A+c_1(\O(m)))=Z(\X,\A)+mZ(X,\alpha)$$ and obtain
	\begin{align}
		\mathcal{Z}(\X,\A+c_1(\O(m)))&=Z(\X,\A+\O(m))-\frac{(\A+c_1(\O(m)))^{n+1}}{(n+1)\alpha^n}Z(X,\alpha), \notag\\
		&=Z(\X,\A)+mZ(X,\alpha)-\frac{\A^{n+1}}{(n+1)\alpha^n}Z(X,\alpha) 	-\frac{(n+1)(\A^n\cup c_1(\O(m)))}{(n+1)\alpha^n}Z(X,\alpha), \notag\\
		&=\mathcal{Z}(\X,\A)  +mZ(X,\alpha) -\frac{(\A^n\cup 	c_1(\O(m)))}{\alpha^n}Z(X,\alpha),   \notag
	\end{align}
	where we have used that $c_1(\O(m))^2=0$ since $\proj^1$ has complex dimension one.
		
	We then need to show that $$mZ(X,\alpha) -\frac{(\A^n\cup c_1(\O(m)))}{\alpha^n}Z(X,\alpha)$$ vanishes. Choosing any $m\pi^*\w_{FS}\in c_1(\O(m))$, we write
	\begin{align}
		\A^n\cup c_1(\O(m))=m\int_\X \pi^*\w_{FS}\wedge\W^n. \notag
	\end{align}
	Because $\int_\X\pi^*\w_{FS}\wedge\W^n=\int_{\X\backslash\X_0}\pi^*\w_{FS}\wedge\W^n$, we now integrate over $X\times(\proj^1\backslash\{0\})$. 
	Using the properties of fibre integrals, $\pi^*\w_{FS}$ factors through the map $\pi:\X\rightarrow\proj^1$, 
	and because of properties $(ii)$ and $(iii)$ of Definition \eqref{test_configuration} we obtain:
	$$\int_\X
	\pi^*\w_{FS}\wedge\W^n=
	\int_{\proj^1}\w_{FS}\int_X\iota_X^*\W^n=
	\int_{\proj^1}
	\w_{FS}
	\int_X\w^n=
	\int_X\w^n,$$ 
	from which we conclude.
\end{proof}

\begin{remark}
	Note that in the special case in which $\Theta=1$, the Chern polynomial is $f(c_1(X))=c_1(X)$, $\rho_j=0$ for all $j\neq n-1$ and $\rho_{n-1}=1$, we recover the Donaldson-Futaki invariant from $\mathcal{Z}(\X,\A)$:
	\begin{align}
		\mathcal{Z}(\X,\A)&=Z(\X,\A)-\frac{\A^{n+1}}{(n+1)\alpha^n}Z(X,\alpha) \notag\\
		&=-\frac{1}{n}\A^{n+1}\cup(c_1(\X)-\pi^*c_1(\proj^1))-\frac{\alpha^{n-1}\cup c_1(X)}{(n+1)\alpha^n}\A^{n+1} \notag\\
		&=\frac{1}{n}\DF(\X,\A). \notag
	\end{align}
\end{remark}

Recall from Equations \eqref{lincomb_centralcharge} and \eqref{lincomb2} that $Z(X,\alpha)$ and $Z(\X,\A)$ can be written as linear combinations in the following way:
\begin{align}
	&Z(X,\alpha)=\sum_{k,l=0}^na_{k,l}(\alpha^l\cup c_1(X)^k\cup\Theta), &\text{and}& &Z(\X,\A)=\sum_{k,l=0}^n\frac{a_{k,l}}{l+1}\DFklt. \notag
\end{align}
As a consequence, we can also write $\Zhat$ as a linear combination over $k$ and $l$. We introduce the numbers
\begin{equation}\label{muklt}
	\mu_{k,l}:=\frac{\alpha^l\cup c_1(X)^k\cup\Theta}{\alpha^n}\in\C,
\end{equation}
from which we define 
\begin{equation}
	\DFhat:=\DFklt-\frac{l+1}{n+1}\mu_{k,l}\A^{n+1}. \notag
\end{equation}
Then from the equality $$\sum_{k,l=0}^na_{k,l}\mu_{k,l}=\frac{Z(X,\alpha)}{\alpha^n}$$ we can deduce that
\begin{equation}
	\mathcal{Z}(\X,\A)= \sum_{k,l=0}^n\frac{a_{k,l}}{l+1}\DFhat.
\end{equation}
	
Now, given any central charge $Z(X,\alpha)$ on $(X,\alpha)$, we are ready to write the corresponding invariant $\Zhat$ as a product of equivariant classes. By linearity, it is enough to do so for the $\DFhat$. We define the following equivariant classes, where $h$ is a hamiltonian for our $S^1$-invariant $\W\in\A$:
\begin{align}
	\mathcal{A}_{\text{eq}}&:=[\W-h], &\text{and}& &\ceq:=[(\RicW-\pi^*\Ric\w_{FS})+(\lap h-\pi^*\lap\hFS)]. \notag
\end{align}
We also recall the equivariant classes $(\Theta_\X)_p$ on $\X$ for any $p\in\{0,\dots,n\}$ from Corollary \ref{corollary_extend_theta}.
	
Finally, from this, we can write:
\begin{equation}\label{DF_localisable}
	\DFhat=\left(\A_{\text{eq}}^{l+1}\cup\ceq^k\cup(\Theta_\X)_{n-k-l}-\frac{l+1}{n+1}\mu_{k,l}\Aeq^{n+1}\right). 
\end{equation}

\begin{remark}\label{eq_class_depends_only_on_kahler_class}
	As our notation suggests, the equivariant class $\mathcal{A}_{\text{eq}}$ does not depend on the choice of representative $\W\in\mathcal{A}$. To see this, recall that if $h$ is a hamiltonian for the $S^1$-action generated by a vector field $V$ with respect to a Kähler metric $\W$, then $h':=h+d^c\psi(V)$ is a moment map for that same action with Kähler metric $\W':=\W+dd^c\psi$. Then $$(\W'-h')-(\W-h)=d_{\text{eq}}(id^c\psi).$$ From this argument, it is clear that $\ceq$ is also independent of the choice of $\W\in\mathcal{A}$. This in turn immediately implies that $Z(\X,\A)$ (and consequently $\Zhat$) is independent of the choice of $\W\in\A$.
\end{remark}
	
\begin{remark}
	If we had not introduced the normalisation from Lemma \ref{renormalisation} and had instead localised $Z(\X,\A)$, by an identical argument we would have similarly obtained $$Z(\X,\A)=\A_{\text{eq}}^{l+1}\cup c_1(\X/\proj^1)_{\text{eq}}^k\cup(\Theta_\X)_{n-k-l}.$$
\end{remark}
	
Now we state the main result of this subsection, for which we fix a test configuration $(\X,\A)$ for $(X,\alpha)$, and write the fixed locus of its $S^1$-action as $$\Fix_{S^1}({\X})=X_{\infty}\sqcup\left(\bigsqcup_{\lambda\in\Lambda}Z_{\lambda}\right),$$ where $\Lambda$ runs over the connected components $Z_{\lambda}$ of $\Fix_{S^1}(\X)$ inside $\X_0$.
	
\begin{proposition}\label{Z-stability_localised}
	The following formula for $\DFhat$ holds:
	\begin{align}
		\DFhat
		&=(-2\pi)^{n+1}\sum_{\lambda\in\Lambda}
		%		\left(
		\int_{Z_{\lambda}}
		\frac{\iota_{\lambda}^*((\W-h)^{l+1}\wedge(\Ric\W+(\lap h-\pi^*\lap\hFS))^k\wedge(\theta_\X)_{n-k-l})}{\euler(N_{\lambda}^\X)} \notag
		\\
		&\hspace{25mm}
		-\frac{l+1}{n+1}\muklt
		\int_{Z_{\lambda}}
		\frac{\iota_{\lambda}^*(\W-h)^{n+1}}{\euler(N_{\lambda}^\X)}.
		%		\right).
		\notag
	\end{align}
\end{proposition}
	
\begin{proof}
	Using Equation \eqref{DF_localisable} for $\DFhat$, we apply Theorem \ref{localisation}:
	\begin{align}\label{everything}
		\DFhat&=\int_{\X}\A_{\text{eq}}^{l+1}\cup\ceq^k\cup(\Theta_\X)_{n-k-l}-\frac{l+1}{n+1}\mu_{k,l}\Aeq^{n+1}, \notag\\
		&=(-2\pi)^{n+1}\sum_{\lambda\in\Lambda}
		\int_{Z_{\lambda}}\frac{\iota_{Z_\lambda}^*\left(\A_{\text{eq}}^{l+1}\cup\ceq^k\cup(\Theta_\X)_{n-k-l}-\frac{l+1}{n+1}\mu_{k,l}\Aeq^{n+1}\right)}
		{\euler(N_{\lambda}^{\mathcal{X}})}\\
		&\hspace{18mm}+(-2\pi)^{n+1}\int_{\X_{\infty}}\frac{\iota_{Z_\infty}^*\left(\A_{\text{eq}}^{l+1}\cup\ceq^k\cup(\Theta_\X)_{n-k-l}-\frac{l+1}{n+1}\mu_{k,l}\Aeq^{n+1}\right)}
		{\euler(N_{\infty}^{\mathcal{X}})}. \notag
	\end{align}
		
	We now check that the contribution from the fibre at infinity, vanishes, i.e. that	$$\int_{\X_{\infty}}\frac{\iota^*_{\infty}\left(\A_{\text{eq}}^{l+1}\cup\ceq^k\cup(\Theta_\X)_{n-k-l}-\frac{l+1}{n+1}\mu_{k,l}\Aeq^{n+1}\right)}{\euler(N_{\infty}^{\mathcal{X}})}=0. \notag$$ 
	First we start by understanding $\euler(N_{\infty}^\X)$. From the biholomorphism $\mathcal{X}\backslash\mathcal{X}_0\cong X\times(\proj^1\backslash\{0\})$ it follows that, for any $t\in\proj^1\backslash\{0\}$, $N_t^{\mathcal{X}}$ is trivial. The $\C^*$-equivariant splitting
	\begin{equation}\label{splitting}
		\iota_t^*T\mathcal{X}\cong T\X_t\oplus\pi^*T\proj^1, \notag
	\end{equation} 
	immediately gives us that $N_t^{\mathcal{X}}\cong\pi^*T_t\proj^1$. Because $\pi$ is $\C^*$-equivariant, and recalling Remark \ref{P1_action}, this means that the weight of the action on $N_{\infty}^\X$ is $-1$. Then,  through Equation \eqref{euler_class}, we conclude that $\euler(N_{\infty}^\X)=1$.
	Now it remains to show that
	\begin{equation}\label{ew}
		\int_{\X_{\infty}}\iota_{\infty}^*\left(\A_{\text{eq}}^{l+1}\cup\ceq^k\cup(\Theta_\X)_{n-k-l}-\frac{l+1}{n+1}\mu_{k,l,\Theta}\Aeq^{n+1}\right)=0. \notag
	\end{equation}
	We will use the subscript notation $\beta_{\infty}=\iota_{\infty}^*\beta$ for any form $\beta$. Using in addition the fact that $\iota_{\infty}^*\pi^*\Ric\w_{FS}=0$, we obtain:
	\begin{align}\label{eww} 
		&\hspace{35mm}\int_{\X_{\infty}}\iota_{\infty}^*\left(\A_{\text{eq}}^{l+1}\cup\ceq^k\cup(\Theta_\X)_{n-k-l}-\frac{l+1}{n+1}\mu_{k,l,\Theta}\Aeq^{n+1}\right)\notag \\
		&=
		\int_{\X_{\infty}}(\W_{\infty}-h_{\infty})^{l+1}\wedge((\Ric\W)_{\infty}+(\lap h-\pi^*\lap\hFS)_{\infty})^k\wedge(\theta_\X)_{n-k-l,\infty} -\frac{l+1}{n+1}\muklt\int_{\X_{\infty}}(\W_{\infty}-h_{\infty})^{n+1},
	\end{align}
	with $\theta_\X\in\Theta_\X$.
	We then use Equation \eqref{laplacian_h} to show that $$(\lap h-\pi^*\lap\hFS)_{\infty}=0.$$ Clearly, by Remark \ref{P1_action}, $(\pi^*\lap h_{FS})_{\infty}=2$. From the same equivariant splitting mentioned above, namely $$\iota_{\infty}^*T\X\cong T\X_{\infty}\oplus\pi^*T_{\infty}\proj^1,$$ and from the fact that $\pi$ is $\C^*$-equivariant, we know that the weight of the $S^1$-action on $N_{\infty}^\X\cong\pi^*T\proj^1$ is the weight of the $S^1$-action on $\proj^1$ at infinity, i.e. $1$. So $\lap h=-2$ as well, hence $(\lap h-\pi^*\lap\hFS)_{\infty}$ vanishes as claimed.
		
	The only remaining terms from Equation \eqref{eww} are
	\begin{align}\label{i_love_mu}
		\int_{\X_{\infty}}(\W_{\infty}-h_{\infty})^{l+1}\wedge\Ric\W^k\wedge(\theta_\X)_{n-k-l,\infty} -\frac{l+1}{n+1}\muklt\int_{\X_{\infty}}(\W_{\infty}-h_{\infty})^{n+1}.
	\end{align}
	Keeping only the $n^{th}$ degree terms and using that $h_{\infty}$ is constant on $\X_{\infty}$, this expression reduces to
	\begin{equation}\label{wait_for_normalisation}
		-(l+1)h_{\infty}\left(\int_{\X_{\infty}}\W_{\infty}^{l+1}\wedge\Ric\W^k\wedge(\theta_\X)_{n-k-l,\infty} -\muklt\int_{\X_{\infty}}\W_{\infty}^{n+1}\right).
	\end{equation}
	We will once again use the splitting $$\iota_{\infty}^*T\mathcal{X}\cong TX_{\infty}\oplus\pi^*T_{\infty}\proj^1.$$ Applying the Whitney sum formula \cite[§24.3]{tu} for vector bundles $E\rightarrow X$ and $F\rightarrow X$ $$c_k(E\oplus F)=\sum_{i=0}^nc_i(E)\cup c_{n-i}(F),$$  and using the fact that $\pi^*T\proj^1$ is trivial gives that $[(\Ric\W)_{\infty}]=[\Ric\W_{\infty}]$. Algebraically, this is equivalent to proving $$K_\X\vert_{\X_0}=K_{\X_0},$$ which can be done through the adjunction formula. 
		
	Finally, the definition of a test configuration tells us that $\A_{\infty}=\alpha$, and from Lemma \ref{extend_theta}, $[\iota_{\infty}^*\theta_\X]=\Theta$. In particular,
	\begin{align}
		&\int_{\X_{\infty}}\W_{\infty}^n=\int_X\w^n &\text{and}& &\int_{\X_{\infty}}\W_{\infty}^l\wedge\Ric\W_{\infty}^k\wedge(\theta_\X)_{n-l-k,\infty}=\int_X\w^l\wedge\Ric\w^k\wedge\theta,
	\end{align}
	so Equation \eqref{i_love_mu} vanishes by definition of the $\muklt$. This shows that the contributions from the fibre at infinity of $\DFhat$ vanish. 
		
	Using this to rewrite Equation \eqref{everything}, keeping only the contributions from the central fibre, and using the fact that $$(\pi^*\Ric\w_{FS})_0=0,$$ we will obtain the statement of the Proposition. 
\end{proof}
	
\begin{remark}
Using Equation \eqref{laplacian_h}, there is a nice way of understanding $$\lap h-\pi^*\lap\hFS$$ on any connected component $Z_{\lambda}$ of the fixed locus of $\X_0$. The arguments are similar to those used for $\X_{\infty}$, only this time the normal bundle $N_{\lambda}^{\X}$ can be of rank $\geq1$, depending on the dimension of $Z_{\lambda}$. From the short exact sequence $$0\longrightarrow N_{\lambda}^{\X_0}\longrightarrow N_{\lambda}^\X\longrightarrow N_{0}^\X\longrightarrow0,$$ we can locally write $N_{\lambda}^\X\cong \iota_{\lambda}^*N_{0}^\X\oplus N_{\lambda}^{\X_0}$. Furthermore, this splitting is equivariant by the equivariance of $\pi:\X\rightarrow\proj^1$. Similarly as for $N_{\infty}^\X$, we find that the weight of the action on $N_0^\X$ is $1$, which is the weight of the $S^1$-action at $0\in\proj^1$. A proof of this is given by Legendre \cite[Theorem 5.3]{Legendre2020}, and we discuss this further in the next subsection.
	Hence Equation \eqref{laplacian_h} tells us that $$\lap h-\pi^*\lap\hFS=\sum_{i=1}^nw_i=\lap h_0,$$ where $i$ runs over the line bundle summands of $N_{\lambda}^{\X_0}$, and the weights $w_i$ are evaluated at each point.
\end{remark}

\begin{remark}
	As explained in Remark \ref{resolution_of_indeterminacy}, Proposition \ref{Z-stability_localised} could be extended to include auxiliary classes $\Theta$ for which the $\Theta_{\X\backslash\X_0}$ constructed in Lemma \ref{extend_theta} does not extend smoothly to $\Theta_\X$ on $\X$, by passing through a resolution of indeterminacy (as described in \cite[Section 2.1]{Dervan2021}).
\end{remark}

\subsection{Localisation for $Z$-critical metrics}
	
We fix a Kähler metric $\w\in\alpha$. Given a central charge $Z(X,\alpha)$ on $(X,\alpha)$, we are interested in the invariants 
\begin{equation}\label{integrals}
	\int_Xh\Ima\left(e^{-i\varphi}\tilde{Z}(X,\w)\right)=\mathcal{F}^Z(V)
\end{equation}
defined in Equation \eqref{Z-futaki}, where $V$ is the generator of an $S^1$-action on $(X,\alpha)$, and $h$ is a hamiltonian for that action with respect to $\w$. In a way analogous to the Futaki invariant in the case of cscK metrics, these invariants give obstructions to the existence of $Z$-critical metrics in the Kähler class $\alpha$. 
	
We would like to write these integrals as a product of equivariant classes, allowing us once again to apply Theorem \ref{localisation}. 
Because of the arbitrary choice of auxiliary cohomology class $\Theta$ involved in the central charge $Z(X,\alpha)$, one cannot reasonably expect to write the invariants \eqref{integrals} as a product of equivariant classes for \textit{any} central charge. From now on, for any $V\in\mathfrak{h}$, we will restrict ourselves to classes $\Theta$ which contain a representative $\theta$ satisfying $\iota_Y\theta=0$ (in particular such a $\theta$ is equivariantly closed), and will only apply localisation for this choice of $\theta$. 
	
By linearity, using Equation \eqref{miam}, we will only need to localise the invariants $\Fklt$. Because the aim of the next subsection we will be to relate these invariants with the $\DFhat$ from the previous section, it will prove useful to normalise the $\Fklt$ in a way analogous to Lemma \ref{renormalisation} in the case of $Z$-stability. 
\noindent The normalisation is as follows: $$\tilde{Z}(X,\w)\rightarrow\tilde{Z}(X,\w)-\frac{Z(X,\alpha)}{\alpha^n}.$$ To make sure this is a reasonable transformation, we must check that it preserves the $Z$-critical equation:
\begin{align}\label{chaton}
	\Ima\left(e^{-i\varphi}\left(\tilde{Z}(X,\w)-\frac{Z(X,\alpha)}{\alpha^n}\right)\right)
	&=\Ima\left(e^{-i\varphi}\tilde{Z}(X,\w)\right)
	-\Ima\left(\frac{e^{-i\varphi}Z(X,\alpha)}{\alpha^n}\right)\notag\\
	&=\Ima\left(e^{-i\varphi}\tilde{Z}(X,\w)\right).
\end{align}

In order to localise the $\Fhat$, we define the following equivariantly closed forms, for any $\w\in\alpha$:
\begin{align}\label{forms_fut}
	&\alpha_{\text{eq}}(Y):=\w-h_Y &\text{and}&  &c_1(X)_{\text{eq}}(Y):=(\Ricw+\lap h_Y). \notag
\end{align}
Because of our requirement that $\theta$ must satisfy $\iota_Y\theta=0$, every term $\theta_p$ in the decomposition
\begin{align}
	\theta=\sum_{p=0}^n\theta_p, \hspace{10mm} \theta_p\in H^{p,p}(X,\C) \notag
\end{align}
is equivariantly closed, so we can consider their equivariant classes $[\theta_p]$.
	
The following results from our discussion and definitions:
\begin{corollary}\label{Z-futaki_localisable}
	The invariants $\Fhat$ can be written
	\begin{equation}
		\Fhat=-\frac{1}{l+1}\Ima\left(e^{-i\varphi}\left(\alpha_{\text{eq}}^{l+1}\cup c_1(X)_{\text{eq}}^k\cup[\theta_{n-l-k}]-\frac{l+1}{n+1}\muklt\alpha_{\text{eq}}^{n+1}\right)\right). \notag
	\end{equation}
\end{corollary}
\begin{proof}
	This is an explicit calculation, in which we choose only the $n^{th}$ degree terms in the integral:
	\begin{align}
		&\hspace{35mm}(\alpha_{\text{eq}}^{l+1}\cup c_1(X)_{\text{eq}}^k\cup[\theta_{n-l-k}]-\frac{l+1}{n+1}\muklt\alpha_{\text{eq}}^{n+1})(V)\hspace{60mm}\notag\\
		\hspace{2mm}&=\int_X\left((\w-h_V)^{l+1}\wedge(\Ric\w+\lap h_V)^k\wedge\theta_{n-k-l}-\frac{l+1}{n+1}\mu_{k,l}(\w-h_V)^{n+1}\right) \notag,\\
		&=-(l+1)\int_Xh_V(\w^{l}\wedge\Ric\w^k\wedge\theta_{n-l-k})+k\int_X\lap h(\w^{l+1}\wedge\Ric\w^{k-1}\wedge\theta_{n-l-k})-(l+1)\mu_{k,l}\int_Xh_V\w^n. \notag
	\end{align}
	The result then follows by applying self-adjointness of the laplacian.
\end{proof}

The following Corollary is proved by Dervan \cite[Corollary 3.12]{Dervan2021} in an indirect way, but in our case follows immediately from Remark \ref{eq_class_depends_only_on_kahler_class}.

\begin{corollary}
	The invariants $\mathcal{F}^Z$ are independent of the choice of $\w\in\alpha$.
\end{corollary}

Once again, the fact that the $\Fhat$ can be written as a product of equivariant classes enables us to apply localisation, implying that the only contributions really come from the fixed locus of the action. So any obstruction to the existence of a $Z$-critical metric in the form of a non-vanishing of $\mathcal{F}^Z(V)$ for some $V\in\mathfrak{h}$ must come from the fixed locus of the corresponding $S^1$-action. 
Explicitly, writing the fixed locus as $$\Fix_{S^1}(X)=\bigsqcup_{\lambda\in\Lambda}Z_{\lambda},$$ the following proposition gives a formula for $\mathcal{F}^Z(V)$.
	
\begin{proposition}\label{Z-metrics_localised}
	Consider a central charge $Z(X,\alpha)$ on $(X,\alpha)$, such that the auxiliary class $\Theta$ contains a representative $\theta$ satisfying $\iota_Y\theta=0$ for an $S^1$-action induced by $V\in\mathfrak{h}$.
	Then we have $$\mathcal{F}^Z(V)=-\frac{(-2\pi)^n}{l+1}\sum_{k,l=0}^n\sum_{\lambda\in\Lambda}\Ima\left(a_{k,l}e^{-i\varphi}\int_{Z_\lambda}\frac{\iota_{\lambda}^*((\w-h_V)^{l+1}\wedge(\Ric\w+\lap h_V)^k\wedge\theta_{n-l-k}-\frac{l+1}{n+1}\mu_{k,l}(\w-h)^{n+1})}{\euler(N_{\lambda}^X)}\right)$$
\end{proposition}
	
\begin{proof}
	From the discussion above, the invariant $\mathcal{F}^Z(V)$ satisfies
	\begin{align}
		\mathcal{F}^Z(V)&=\sum_{k,l=0}^n\mathcal{F}^Z_{k,l}(V),\notag\\
		&=-\frac{1}{l+1}\sum_{k,l=0}^n\Ima\left(a_{k,l}e^{-i\varphi}\left(\alpha_{\text{eq}}^{l+1}\cup c_1(X)_{\text{eq}}^k\cup[\theta_{n-l-k}]-\frac{l+1}{n+1}\muklt\alpha_{\text{eq}}^{n+1}\right)\right). \notag
	\end{align}
	By linearity, we only need to show that
	\begin{align}
		\Fhat=-\frac{(-2\pi)^n}{l+1}\sum_{\lambda\in\Lambda}\int_{Z_{\lambda}}\frac{\iota_{\lambda}^*((\w-h_V)^{l+1}\wedge(\Ric\w+\lap h_Y)^k\wedge\theta_{n-l-k}-\frac{l+1}{n+1}\muklt(\w-h_V)^{n+1})}{\euler(N_{\lambda}^X)}, \notag
	\end{align}
	which is a direct consequence of Theorem \ref{localisation} applied to the equivariant class of $\mathcal{F}^Z_{k,l}$ given in Corollary \ref{Z-futaki_localisable}.
\end{proof}

\begin{remark}
	Note that by Equation \eqref{chaton}, we could equivalently have written $$\mathcal{F}^Z_{k,l}=-\frac{1}{l+1}\Ima\left(e^{-i\varphi}\left(\alpha_{\text{eq}}^{l+1}\cup c_1(X)_{\text{eq}}^k\cup[\theta_{n-l-k}]\right)\right),$$ which would have given the (simpler) formula
	\begin{align}
		\mathcal{F}^Z(V)=-\frac{(-2\pi)^n}{l+1}\sum_{k,l=0}^n\sum_{\lambda\in\Lambda}\Ima\left(a_{k,l}e^{-i\varphi}\int_{Z_\lambda}\frac{\iota_{\lambda}^*((\w-h_V)^{l+1}\wedge(\Ric\w+\lap h_V)^k\wedge\theta_{n-l-k})}{\euler(N_{\lambda}^X)}\right). \notag
	\end{align}
	This, of course, is equivalent by virtue of Equation \eqref{chaton}. Our motivation for adding the normalisation term will become apparent in the following section, where we will link central charges $Z(\X,\A)$ on a test configuration with the quantity $\mathcal{F}^Z(V)$ on its central fibre, where $V$ is the generator of the $S^1$-action on the test configuration. Recall from the previous section that the normalisation $Z(\X,\A)\rightarrow\Zhat$ was necessary for the contribution from the fibre at infinity to vanish (see proof of Proposition \ref{Z-stability_localised}). Then, our normalisation $$\tilde{Z}(X,\w)\rightarrow \tilde{Z}(X,\w)-\frac{Z(X,\alpha)}{\alpha^n}$$ is the corresponding normalisation to make on the central fibre. All of this will be made more precise in the proof of Theorem 
	\ref{central_fibre_theorem}.
\end{remark}
	
From Proposition \ref{Z-metrics_localised}, we can deduce an immediate corollary, which makes explicit a way in which the choice of $\theta\in\Theta$ can affect the use of the invariants $\mathcal{F}^Z$ as obstructions to the existence of $Z$-critical metrics.
	
\begin{corollary}
	The invariants $$\Fklt=\int_Xh\Ima\left(e^{-i\varphi}\left(\frac{\Ric\w^k\wedge\w^l\wedge\theta}{\w^n}-\frac{k}{l+1}\lap\frac{\Ric\w^{k-1}\wedge\w^{l+1}\wedge\theta}{\w^n}\right)\right)\w^n$$ vanish if $\theta_{n-l-k}$ is equivariantly exact, that is if $\theta$ is exact or $\theta=\iota_V\beta$ for some form $\beta$.
\end{corollary}
\begin{proof}
	If $\theta_{n-l-k}$ is equivariantly exact, it is equivariantly closed, and thus $\Fklt$ can be written as a product of equivariant classes. Then, by linearity in $\theta_{n-l-k}$, any $\Fklt$ must vanish for any equivariantly exact $\theta_{n-l-k}$.
\end{proof}

\subsection{The smooth central fibre case}\label{smoot_central_fiber_case}

Let $(\X,\A)$ be a test configuration for a Kähler manifold $(X,\alpha)$. We are interested in the case where the central fibre $(X_0,\A_0)$ is smooth, in which case $\A_0$ is a Kähler metric on $X_0$ and we fix a metric $\W_0=\iota_0^*\W$ in $\A_0$. In this specific case, we consider the $S^1$-action on $X_0$ induced by the $S^1$-action on $\X$. We denote its fundamental vector field by $V_0$, and a hamiltonian for this vector field $h_0$ with respect to $\W_0$. Of course, these are restrictions to $\X_0$ of the fundamental vector field $V$ and the hamiltonian $h$ with metric $\W$ of the $S^1$-action on $\X$.
	
\begin{lemma}
	Let $\Theta$ be an auxiliary class on a Kähler manifold $(X,\alpha)$, and $(\X,\A)$ be any test configuration for $(X,\alpha)$. If the central fibre $(\X_0,\A_0)$ is smooth, then it inherits $S^1$-equivariant auxiliary classes $(\Theta_0)_p$ for each $p\in\{0,\dots,n\}$ in the decomposition $$\Theta_0=\oplus_{p=0}^n(\Theta_0)_p, \hspace{3mm} (\Theta_0)_p\in H^{p,p}(X,\C),$$ where the equivariance is with respect to the $S^1$-action generated by $V_0$ on $\X_0$. 
\end{lemma}
\begin{proof}
	From any auxiliary class $\Theta$ on $(X,\alpha)$, we use Lemma \ref{extend_theta} to obtain an $S^1$-equivariant class $\Theta_\X$ on $\X$. Corollary \ref{corollary_extend_theta} tells us that each term $(\Theta_0)_p$ in the decomposition is equivariantly closed. For any $\theta_\X\in\Theta_\X$, we consider $$\theta_0:=\iota_0^*\theta_\X.$$ 
	Clearly, each term $(\theta_0)_{n-l-k}$ is equivariantly closed, hence defines an $S^1$-equivariant class, with the $S^1$-action on $\X_0$ generated by $V_0$. 
	It now remains to show that the equivariant class $\Theta_0$ of $\theta_0$ is independent of the choice of $\theta_\X\in\Theta_\X$. This follows from the fact that $\iota_0^*(d_{\text{eq}})_\X=(d_{\text{eq}})_{\X_0}\iota_0^*$, where $(d_{\text{eq}})_\X=d-\iota_V$ on $\X$, and $(d_{\text{eq}})_{\X_0}=d-\iota_{V_0}$ on $\X_0$.
\end{proof}

We continue our previous notation and write the fixed locus of this action as a union of its connected components as: $$\Fix_{S^1}(X)=\bigsqcup_{\lambda\in\Lambda}Z_{\lambda}.$$

This subsection aims to give a new, more direct proof of a result proved by Dervan \cite[Proposition 3.11]{Dervan2021} using the localisation formulas \eqref{Z-metrics_localised} and \eqref{Z-stability_localised} from the two previous subsections:
	
\begin{theorem}\label{central_fibre_theorem}
	Let $(\X,\A)$ be a test configuration for a Kähler manifold $(X,\alpha)$, with smooth central fibre $(\X_0,\A_0)$. A central charge $Z(X,\alpha)$ induces a central charge $Z(\X_0,\A_0)$ on $\X_0$ by taking $\Theta_0:=[\iota_{\X_0}^*\theta_\X]$. We write $\varphi_0$ for its argument. 
	Then, for any $\theta_0\in\Theta_0$, $$\Ima\left(e^{-i\varphi}\tilde{Z}(\X,\A)\right)=-2\pi\mathcal{F}^{Z(\X_0,\A_0)}(V_0)=-2\pi\int_{\X_0}h_0\Ima\left(e^{-i\varphi_0}\tilde{Z}(\X_0,\A_0)\right).$$
\end{theorem}
	
The style of proof we will use was used by Legendre in \cite[Theorem 5.3]{Legendre2020} in the case of K-stability and cscK metrics. In particular, our proof relies heavily on the following lemma, which is proved by Legendre in her proof of the theorem just mentioned:
	
\begin{proposition}\cite[Theorem 5.2]{Legendre2020}\label{legendre_prop}
	The following hold for any test configuration $(\X,\A)$ for a Kähler manifold $(X,\alpha)$ with smooth central fibre:
	\begin{enumerate}
		\item the line bundle $N_0^\X$ is topologically trivial, which in particular implies $\iota_{\lambda}^*[\Ric\W]=\iota_{\lambda}^*[\Ric\W_0]$;
		\item for any $Z_{\lambda}$, there is a local splitting $$N_{\lambda}^{\X}\cong\iota^*_0N_{\X_0}^{\X}\oplus N_{\lambda}^{\X_0},$$ for which the weight of the action on the line bundle $\iota_{\lambda}^*N_0^\X$ is $1$;
		\item $\euler(N_0^{\mathcal{X}})=-\euler(N_0^{X_0}).$
	\end{enumerate}
\end{proposition}
	
\begin{remark}
	The third fact actually follows simply from the first two. From the short exact sequence $$0\longrightarrow N_{\lambda}^{\X_0}\longrightarrow N_{\lambda}^\X\longrightarrow N_{X_0}^\X\longrightarrow 0,$$ we apply the the duality formula \cite[§4]{zeroes} $$e(N_Z^{\mathcal{X}})=e(N_Z^{\X_0})\cup\iota_Z^*e(N_0^{\mathcal{X}}).$$ Applying Equation \eqref{euler_class} and using the first two facts, we can see that $e(N_0^{\X})=-1$, which implies the result. 
\end{remark}
	
\begin{remark}
	Note that the actual result proved by Legendre was $$\euler(N_0^{\mathcal{X}})=-\frac{\euler(N_0^{X_0})}{2\pi},$$ and that the difference here is in the convention used in the localisation formula \ref{localisation} and in the definition of the equivariant Euler class \eqref{euler_class}.
\end{remark}

\begin{proof}[Proof of Theorem \ref{central_fibre_theorem}]
	We start by noticing that $$\varphi_0=\varphi.$$ This follows from Ehresmann's Lemma and from the fact that $$\int_{\X_t}\iota_t^*\tilde{Z}(\X_t,\W_t)$$ is constant and equal to $Z(X,\alpha)$ for any $t\in(\proj^1\backslash\{0\})$  (this follows from the construction of $\Theta_{\X\backslash\X_0}$ in Lemma \ref{extend_theta}). Hence $$\int_{\X_0}\tilde{Z}(\X_0,\W_0)=Z(X,\alpha).$$
	Then, by linearity, it is sufficient to show that, for any $k,l=0\dots,n$, $$\frac{1}{l+1}\Ima\left(e^{-i\varphi}\DFhat\right)=\mathcal{F}_{k,l}^{Z(\X_0,\A_0)}(V_0).$$
	From Proposition \ref{Z-stability_localised}, we write
	\begin{align}\label{hum}
		\DFhat
		&=(-2\pi)^{n+1}\sum_{\lambda\in\Lambda}
		%		\left(
		\int_{Z_{\lambda}}
		\frac{\iota_{\lambda}^*((\W-h)^{l+1}\wedge(\Ric\W+(\lap h-\pi^*\lap\hFS))^k\wedge(\theta_\X)_{n-k-l})}{\euler(N_{\lambda}^\X)} \notag
		\\
		&\hspace{25mm}
		-\frac{l+1}{n+1}\muklt
		\int_{Z_{\lambda}}
		\frac{\iota_{\lambda}^*(\W-h)^{n+1}}{\euler(N_{\lambda}^\X)}.
		%		\right).
	\end{align}
	We will now use the fact that $\iota_{\lambda}^*=\iota_{\lambda}^*\iota_0^*$, combined with Proposition \ref{legendre_prop} to get 
	\begin{align}
		&\iota_{\lambda}^*(\W-h)=\iota_{\lambda}^*(\W_0-h_0) &\mbox{and}\hspace{7mm} &\iota_{\lambda}^*(\Ric\W+(\lap h-\pi^*\lap\hFS))=\iota_{\lambda}^*(\Ric\W_0+((\lap h)_0-(\pi^*\lap\hFS)_0)). \notag
	\end{align}
	Let us study in more detail the term $$(\lap h-\pi^*\lap\hFS)_0.$$ 
	At any point $z\in Z_{\lambda}$, the normal bundle $N_{\lambda}^\X$ locally splits as $$N_{\lambda}^\X\cong N_{\lambda}^{\X_0}\oplus\iota_{\lambda}^*N_{\X_0}^{\X}.$$ 
	Point $(2)$ of Proposition \ref{legendre_prop} tells us that the weight coming from the $N_{\X_0}^\X$ component is $1$, which means from Equation \eqref{laplacian_h} that $$(\lap h)_0=-2\left(1+\sum_{i=1}^{n-n_{\lambda}}w_i\right),$$ where $i$ runs over the line bundle components of $\iota_{\lambda}^*N_{\lambda}^{\X_0}$. 
	Because $$(\pi^*\lap\hFS)_0=-2$$ at any point $z\in Z_{\lambda}$ by Remark \ref{P1_action}, $$(\lap h)_0-(\pi^*\lap\hFS)_0=-2\sum_{i=1}^{n-n_{\lambda}}w_i=\lap h_0.$$
	
%	Finally, we notice that $$\int_{\X_0}\iota_0^*\theta_{\X}=\int_X\theta$$ for $\theta_\X\in\Theta_\X$ and $\theta\in\Theta$. This follows from Ehresmann's Lemma, and from the fact that the integral $$\iota_t^*\Theta_{\X}([\X_t])$$ viewed as a function of $t\in(\proj^1\backslash\{0\}$ is constant and equal to $$\Theta([X])$$ in a neighbourhood of $0\in\proj^1$ (because of the construction of $\Theta_\X$ in Lemma \ref{extend_theta}). In particular, for any $\lambda\in\Lambda$, $$\int_{Z_{\lambda}}\iota_{\lambda}\theta_\X=\int_{Z_{\lambda}}\iota^*_{\lambda}\iota^*_0\theta_\X=$$
		
Combining all of this with point 3) of Proposition \ref{legendre_prop} we may rewrite Equation \eqref{hum} as
		
	\begin{align}\label{hm}
		\DFhat
		&=-(-2\pi)^{n+1}\sum_{\lambda\in\Lambda}
		\int_{Z_{\lambda}}
		\frac{\iota_{\lambda}^*((\W_0-\mu_0)^{l+1}\wedge(\Ric\W_0+\lap\mu_0)^k\wedge(\theta_\X)_{n-k-l})}{\euler(N_{\lambda}^{\X_0})} \notag
		\\
		&\hspace{25mm}
		-\frac{l+1}{n+1}\muklt
		\int_{Z_{\lambda}}
		\frac{\iota_{\lambda}^*(\W_0-\mu_0)^{n+1}}{\euler(N_{\lambda}^{\X_0})} \notag, \\
		&=2\pi\left((\A_0)_{\text{eq}}^{l+1}\cup c_1(\X_0)_{\text{eq}}^k\cup \Theta_{0,n-k-l} -\frac{l+1}{n+1}\mu_{k,l}(\A_0)^{n+1}_{\text{eq}}\right)(V_0). \notag
	\end{align}
	Now, recalling Proposition \ref{Z-metrics_localised}, we obtain:
	\begin{align}
		\Ima\left(\frac{\DFhat}{Z(X,\alpha)}\right)&=\frac{2\pi}{r(X,\alpha)}\Ima\left(e^{-i\varphi}\left((\A_0)_{\text{eq}}^{l+1}\cup c_1(\X_0)_{\text{eq}}^k\cup \Theta_{0,n-k-l} -\frac{l+1}{n+1}\mu_{k,l}(\A_0)^{n+1}_{\text{eq}}\right)(V_0)\right) \notag\\
		&=-2\pi\mathcal{F}^{Z(\X_0,\A_0)}_{k,l}(V_0),
	\end{align}
	as desired.
\end{proof}

As mentioned in the introduction, an important implication of this proof is that equivariant localisation can be used to derive the $Z$-critical equation directly from the corresponding algebro-geometric stability condition. In particular, this suggests a way of extending the theory to central charges involving higher Chern classes of $(X,\alpha)$, by extending a representative of $c_k(X)$ to an equivariantly closed form and considering its equivariant class. 
	
\printbibliography

\end{document}